\RequirePackage{fix-cm}
\documentclass[smallextended, envcountsame]{svjour3}       
\smartqed  
\usepackage{graphicx}
\usepackage{mathptmx}      
\usepackage{latexsym}
\usepackage{amssymb}
\usepackage{amsmath}
\usepackage{enumerate}

\numberwithin{equation}{section}
\numberwithin{theorem}{section}

\newcommand{\ints}{\mathbb Z}
\newcommand{\rats}{\mathbb Q}
\newcommand{\complex}{\mathbb C}
\newcommand{\proj}{\mathbb P}
\newcommand{\Gp}{\mathbb G}
\newcommand{\Gal}{\text{Gal}}
\newcommand{\Spec}{\mbox{Spec }}
\newcommand{\mc}[1]{\mathcal{#1}}
\newcommand{\ol}[1]{\overline{#1}}
\newcommand{\fracpart}[1]{\langle #1 \rangle}

\journalname{}
\begin{document}

\title{On Colmez's product formula for periods of CM-abelian varieties
\thanks{The author was supported by an NSF Postdoctoral Research Fellowship in the Mathematical Sciences.  Final editing of this work
was undertaken during a fellowship at the Max Planck Institute for Mathematics in Bonn.}
}


\author{Andrew Obus}


\institute{
               Columbia University, Department of Mathematics, 2990 Broadway, MC 4403, New York, NY 10027\\
              Tel.: +1-212-854-6240\\
              Fax: +1-212-854-8962\\
              \email{andrewobus@gmail.com}  }

\date{Received: date / Accepted: date}

\maketitle

\begin{abstract}
Colmez conjectured a product formula for periods of abelian varieties with complex multiplication by a field $K$, analogous to the standard 
product formula in algebraic number theory.  He proved this conjecture up to a rational power of $2$ for $K/\rats$ abelian.  
In this paper, we complete the proof of Colmez for $K/\rats$ abelian by eliminating this
power of $2$.  Our proof relies on analyzing the Galois action on the De Rham cohomology of Fermat curves in mixed characteristic $(0,2)$, which in turn
relies on understanding the stable reduction of $\ints/2^n$-covers of the projective line, branched at three points.
\end{abstract}

\section{Introduction}\label{Sintro}
The \emph{product formula} in algebraic number theory states that, given an algebraic number $x \ne 0$ in a number field $K$, 
the product of $|x|$ as $|\cdot|$ ranges over all
inequivalent absolute values of $K$ (appropriately normalized) is equal to $1$.  In logarithmic form, the sum of $\log |x|$ as $|\cdot|$ ranges over all inequivalent
absolute values is $0$.
In \cite{Co:pv}, Colmez asked whether an analogous product formula might hold for periods of 
algebraic varieties, and conjectured that it would hold for periods of abelian varieties with complex multiplication (CM-abelian varieties).  He proved that, for abelian
varieties with complex multiplication by \emph{abelian} extensions of $\rats$, such a product formula holds (in logarithmic form) up to an (unknown) rational multiple of 
$\log 2$ (\cite[Th\'{e}or\`{e}me
0.5 and discussion after Conjecture 0.4]{Co:pv}).  A key step in this proof was provided by work of Coleman and McCallum (\cite{CM:sr}, \cite{Co:fm}) on 
understanding stable
models of quotients of Fermat curves in mixed characteristic $(0,p)$, where $p$ is an odd prime.  These quotients are $\ints/p^n$-covers of the projective line, branched 
at three points.  The unknown rational multiple of $\log 2$ was necessary in \cite{Co:pv} precisely because the stable models of $\ints/2^n$-covers of the 
projective line, branched at three points, in mixed characteristic $(0, 2)$, were not well-understood at the time.  This problem was solved by the author in  
\cite{Ob:fm}, where a complete description of the stable models of such covers was given.
In this paper, we use the results of \cite{Ob:fm} to complete the proof of Colmez's product formula for abelian extensions of $\rats$ by 
eliminating the multiple of $\log 2$ in question.

Colmez first looks at the example of $2 \pi i$, which is a period for the variety $\Gp_m$, rather than for an abelian variety.  
For each prime $p$, one can view $2\pi i$ as an element $t_p$ 
of Fontaine's ring of periods $\mathbf{B}_p$, and its $p$-adic absolute value is $|t_p|_p = p^{1/(1-p)}$.   The archimedean absolute value $|\cdot|_{\infty}$ is the 
standard one, so $|2 \pi i|_{\infty}  = 2\pi$.  The logarithm of the product of all of these absolute values is 
$$\log 2\pi - \sum_{p < \infty} \frac{\log p}{p-1}.$$   This sum does not converge, but formally, it is equal to $\log 2\pi - \frac{\zeta'(1)}{\zeta(1)}$, where $\zeta$ is the 
Riemann zeta function.  Using the functional equation of $\zeta$ (and ignoring the $\Gamma$ factors), we obtain $\log 2 \pi - \frac{\zeta'(0)}{\zeta(0)}$, which is equal 
to $0$.  In this sense, we can say that the product formula holds for $2 \pi i$.

The above method can be adapted to give a definition of what it means to take the logarithm of the product of all the absolute values of a period, 
and thus to give a product formula meaning.  
Many subtleties arise, and the excellent and thorough introduction to \cite{Co:pv} discusses them in detail.  We will not attempt to recreate this discussion.  
Instead, we will just note that Colmez shows that the product formula for periods of CM-abelian varieties with complex multiplication by an abelian extension of $\rats$
(in logarithmic form) is equivalent to the formula 
\begin{equation}\label{Eproductformula}
ht(a) = Z(a^*, 0)
\end{equation}
for all $a \in \mc{CM}^{ab}$ (\cite[Th\'{e}or\`{e}me II.2.12(iii)]{Co:pv}).  
Here, $\mc{CM}^{ab}$ is the vector space of $\rats$-valued, locally constant functions $a: \Gal(\rats^{ab}/\rats) \to \rats$ such that,
if $c$ represents complex conjugation, then $a(g) + a(cg)$ does not depend on $g \in G_{\rats}$.  
Such a function can be decomposed into a $\complex$-linear combination of Dirichlet
characters whose L-functions do not vanish at $0$.  If $a \in \mc{CM}^{ab}$ then we define $a^* \in \mc{CM}^{ab}$ by $a^*(g) = a(g^{-1})$.  Also, $Z(\cdot, 0)$ is the 
unique $\complex$-linear function on 
$\mc{CM}^{ab} \otimes \complex$ equal to $\frac{L'(\chi, 0)}{L(\chi, 0)}$ when its argument is a Dirichlet character $\chi$ whose $L$-function does not vanish at $0$.  
Lastly, $ht(\cdot)$ is a $\complex$-linear function on 
$\mc{CM}^{ab} \otimes \complex$ related to Faltings heights of abelian varieties (see \cite[Th\'{e}or\`{e}me 0.3]{Co:pv} for a precise definition, also
\cite{Ya:cs}).

Colmez shows (\cite[Proposition III.1.2, Remarque on p.\ 676]{Co:pv}) that 
\begin{equation}\label{Eerror}
Z(a^*, 0) - ht(a) = \sum_{p \text{ prime}} w_p(a) \log p,
\end{equation} where $w_p: \mc{CM}^{ab} \to \rats$ is a $\rats$-linear function (depending on $p$) that will be defined in
\S\ref{Sfrob}.  He then further shows that $w_p(a) = 0$ for all $p \geq 3$ and all $a \in \mc{CM}^{ab}$ (\cite[Corollaire III.2.7]{Co:pv}).  
Thus (\ref{Eproductformula}) is correct up to a 
adding a rational multiple of $\log 2$. 
Our main theorem (Theorem \ref{Tmain}) states that $w_2(a) = 0$ for all $a \in \mc{CM}^{ab}$, thus proving (\ref{Eproductformula}).

We note that, in light of the expression (\ref{Eproductformula}), Colmez's formula is fundamentally about relating periods of CM-abelian varieties to 
logarithmic derivatives of L-functions.  That this can be expressed as a product formula is aesthetically pleasing, 
but the main content is encapsulated by (\ref{Eproductformula}).

In \S\ref{Sfrob}, we define $w_p$ and show how it is related to De Rham cohomolgy of Fermat curves.  
In \S\ref{Smonodromy}, we write down the important properties of the stable model of a certain quotient of 
the Fermat curve $F_{2^n}$ of degree $2^n$ ($n \geq 2$) over $\rats_2$, and we discuss the monodromy action on the stable reduction.  
In \S\ref{Sdiffforms}, we show how knowledge of 
this stable model, along with the monodromy action, 
allows us to understand the Galois action on the De Rham cohomology of $F_{2^n}$.  In \S\ref{Sformula}, we show how this is used to 
prove that $w_2(a) = 0$.  Lastly, in \S\ref{Scomputations}, we collect some technical power series computations that are used in 
\S\ref{Sdiffforms}, but would interrupt the flow of the paper if included there.

\subsection{Conventions}\label{Sconventions}
The letter $p$ always represents a prime number.  If $x \in \rats/\ints$, then $\fracpart{x}$ is the unique representative for $x$ in the interval $[0, 1)$.  
The standard $p$-adic valuation on $\rats$ is denoted $v_p$, and the subring $\ints_{(p)} \subseteq \rats$ 
consists of the elements $x \in \rats$ with $v_p(x) \geq 0$.  If $K$ is a field, then $\ol{K}$ is its algebraic closure and $G_K$ is its absolute Galois group.

\section{Galois actions on De Rham cohomology}\label{Sfrob}
The purpose of this section is to define the function $w_p(a)$ from (\ref{Eerror}).  In order to make this definition, one must first consider a particular
rational factor of the Jacobian of the $m$th Fermat curve (where $m$ is related to $a$).  
This factor will have complex multiplication, and we will choose a de Rham
cohomology class that is an eigenvector for this complex multiplication.  One can then define the ``$p$-adic valuation" 
of such a cohomology class, and this valuation essentially determines $w_p(a)$.

Recall that the action of $G_{\rats}$ on roots of unity gives a
homomorphism $\chi: G_{\rats} \to \hat{\ints}^{\times}$.  This factors through $\Gal(\rats^{ab}/\rats)$, giving an isomorphism
$\Gal(\rats^{ab}/\rats) \cong \hat{\ints}^{\times}$, called the 
\emph{cyclotomic character}.  Multiplication by the cyclotomic character gives a well-defined action of
$G_{\rats}$ on $\rats/\ints$, factoring through $\Gal(\rats^{ab}/\rats)$.  

The following definitions are from \cite[III]{Co:pv}.  
Recall that $\mc{CM}^{ab}$ is the vector space of $\rats$-valued, locally constant functions $a: \Gal(\rats^{ab}/\rats) \to \rats$ such that, 
if $c$ represents complex conjugation, then $a(g) + a(cg)$ does not depend on $g \in G_{\rats}$ (\cite[p.\ 627]{Co:pv}).  
For $r \in \rats/\ints$, define an element $a_r \in \mc{CM}^{ab}$ by
$$a_r(g) = \fracpart{gr} - \frac{1}{2}.$$  One can show that the $a_r$ generate $\mc{CM}^{ab}$ as a $\rats$-vector space.  
For $r \in \rats/\ints$, set $v_p(r) = \min(v_p(\fracpart{r}), 0)$, and set 
\begin{equation}\label{Epdef}
r_{(p)} = p^{-v_p(r)}r \in \rats/\ints.
\end{equation}
Set
$$V_p(r) = \begin{cases} 0 & r \in \ints_{(p)}/\ints \\ (\fracpart{r} - \frac{1}{2})v_p(r) - \frac{1}{(p-1)p^{-v_p(r)-1}}(\fracpart{\frac{r_{(p)}}{p}} - \frac{1}{2}) & \text{otherwise,}
\end{cases}$$
where $\frac{r_{(p)}}{p}$ is the unique element of $\ints_{(p)}/\ints$ such that $\frac{r_{(p)}}{p} \cdot p = r_{(p)}$.

Let $q = (\rho, \sigma, \tau) \in (\rats/\ints)^3,$ such that $\rho+\sigma+\tau = 0$ and none of $\rho$, $\sigma$, or $\tau$ are $0$.  
Let $m$ be a positive integer such that $m\rho = m\sigma = m\tau = 0$.
Let $\epsilon_q  = \fracpart{\rho} + \fracpart{\sigma} + \fracpart{\tau} - 1$.
Let $F_m$ be the $m$th Fermat curve, that is, the smooth, proper model of the affine curve over $\rats$ given by $u^{m} + v^{m} = 1$, and let $J_m$ be its Jacobian. 
Write $\fracpart{\rho} = \frac{a}{m}$ and $\fracpart{\sigma} = \frac{b}{m}$.
Consider the closed differential form 
$$\eta_{m, q} := m\fracpart{\rho+\sigma}^{\epsilon_q}u^av^b\frac{v}{u}d\left(\frac{u}{v}\right)$$ on $F_m$. 
We can view its De Rham cohomology class as a class $\omega_{m, q} \in H^1_{DR}(J_m) \cong H^1_{DR}(F_m)$ over $\rats$.  
It turns out that there is a particular rational factor $J_q$ of $J_m$ with complex
multiplication, and a class $\omega_q \in H^1_{DR}(J_q)$, such that the pullback of $\omega_q$ to $J_m$ is $\omega_{m, q}$.  Furthermore, $\omega_q$ is an
eigenvector for the complex multiplication on $J_q$.  As is suggested by the notation, the pair 
$(J_q, \omega_q)$ depends only on $q$, not on $m$, up to isomorphism (\cite[p.\ 674]{Co:pv}).  

Now, $G_{\rats}$ acts diagonally on $(\rats/\ints)^3$ by the cyclotomic character.  If $\gamma \in G_{\rats}$, then 
$J_q = J_{\gamma q}$ (\cite[p.\ 674]{Co:pv}).  Fix an embedding $\ol{\rats} \hookrightarrow \ol{\rats}_p$, which gives rise to an embedding $G_{\rats_p} \hookrightarrow
G_{\rats}$.  If $\gamma$ lies in the inertia group $I_{\rats_p} \subseteq G_{\rats_p} \subseteq G_{\rats}$, then $\gamma$ acts on $J_q$, and thus on 
$H^1_{DR}(J_q, \ol{\rats}_p)$.
We have $\gamma^*\omega_q = \beta_{\gamma}(q) \omega_{\gamma q}$ where the constant $\beta_{\gamma}(q)$ lies in some finite extension of 
$\rats_p$ (\cite[pp.\ 676-7]{Co:pv}).  
We also note that $I_{\rats_p}$ acts on $H^1_{DR}(F_m, \ol{\rats}_p) \cong H^1_{DR}(J_m, \ol{\rats}_p)$ via its action on $F_m$.  One derives
\begin{equation}\label{Epullback}
\gamma^*\omega_{m, q} = \beta_{\gamma}(q) \omega_{m, \gamma q}.
\end{equation}

If $K$ is a $p$-adic field with a valuation $v_p$, then there is a notion of $p$-adic valuation of $\omega \in H^1_{DR}(A)$ whenever $A$ is a CM-abelian variety 
defined over $K$ and $\omega$ is an eigenvector for the complex multiplication (\cite[p.\ 659]{Co:pv}---note that $\omega_q$ is such a class).  
By abuse of notation, we also write this valuation as $v_p$.
It has the property that, if $c \in K$, then $v_p(c\omega) = v_p(c) + v_p(\omega)$.  

\begin{lemma}\label{Laux}
If $\gamma \in I_{\rats_p}$, then $v_p(\omega_q) - v_p(\omega_{\gamma q}) = v_p(\beta_{\gamma}(q))$.
\end{lemma}

\begin{proof}
By \cite[Th\'{e}or\`{e}me II.1.1]{Co:pv}, we have $v_p(\omega_q) = v_p(\gamma^*\omega_q)$.  The lemma then follows from the definition of
$\beta_{\gamma}(q)$.
\end{proof}

Let $b_q = a_{\rho} + a_{\sigma} + a_{\tau} \in \mc{CM}^{ab}$.
There is a unique linear map $w_p: \mc{CM}^{ab} \to \rats$ such that $$w_p(b_q) = v_p(\omega_q) - V_p(q)$$ (\cite[Corollaire III.2.2]{Co:pv}).  This is 
the map $w_p$ from (\ref{Eerror}).  Recall from \S\ref{Sintro}
that Colmez showed $w_p(a) = 0$ for all $p \geq 3$ and all $a \in \mc{CM}^{ab}$.  
In Theorem \ref{Tmain}, we will show that $w_2(a) = 0$ for all $a \in \mc{CM}^{ab}$.  

\section{Fermat curves}\label{Sfermat}

In general, a branched Galois cover $f: Y \to X := \proj^1$ defined over $\rats_p$ does not necessarily have good reduction.  However, assuming that 
$2g(X) + r \geq 3$ (where $r$ is the number of branch points), one can always find a finite extension $K/\rats_p$ with valuation ring $R$, and a
\emph{stable model} $f^{st}: Y_R \to X_R$ for the cover (i.e., $f^{st}$ is a finite map of flat $R$-curves whose generic fiber
is $f$, and where $Y^{st}$ has reduced, stable fibers, considering the specializations of the ramification points of $f$ as marked points).
The special fiber $\ol{f}: \ol{Y} \to \ol{X}$ of $f^{st}$ is called the \emph{stable reduction} of the cover.
Furthermore, there is an action of $G_{\rats_p}$ on $\ol{f}$ (called the \emph{monodromy action}), 
given by reducing its canonical action on $f$, and this action factors through $\Gal(K/\rats_p)$.  For more details, see 
\cite{DM}, \cite{Ra}, \cite{Liu}.

Calculating the stable reduction and monodromy action of a cover can be difficult, even when the Galois group is simple  
(see, e.g., \cite{LM}, where Lehr and Matignon calculate the stable reduction and monodromy action of $\ints/p$-covers branched at
arbitrarily many equidistant points).  Restricting to three branch points can simplify matters.  A major result of Coleman and McCallum (\cite{CM:sr})
calculated the stable reduction of all cyclic covers of $\proj^1$ branched at three points, when $p \neq 2$.  From this, the monodromy 
action was calculated in \cite{Co:fm}, which sufficed to prove Colmez's product formula up to the factor of $\log 2$.  The case $p=2$ (for
three-point covers) is somewhat more complicated, and requires new techniques by the author in \cite{Ob:fm}.
Enough details are given in \cite{Ob:fm} to calculate the monodromy action explicitly, which 
we do to the extent we need to in \S\ref{Smonodromy}.
The work in \S\ref{Sdiffforms} and \S\ref{Sformula} mimics the work in \cite{Co:fm} and \cite{Co:pv}, respectively, 
to show how a knowledge of the monodromy action leads to a proof of the product formula.

\subsection{The monodromy action}\label{Smonodromy}

Fix $n \geq 2$.  Let $f:Y \to X := \proj^1$ be the 
branched cover given birationally by the equation $y^{2^n} = x^a(x-1)^b$, defined over $\rats_2$, where $x$ is a fixed coordinate on $\proj^1$.
Assume for this entire section that $a$ is odd, that $1 \leq v_2(b) \leq n-2$, and that $0 < a, b < 2^n$.  
Set $s = n-v_2(b)$ (this makes $2^s$ the branching index of $f$ at $x=1$).  
Thus $s \geq 2$.  Let $K/\rats_2$ be a finite extension, with valuation ring $R$, over which $f$ admits a stable model $f^{st}: Y^{st} \to X^{st}$.
Let $k$ be the residue field of $K$.  We write $I_{\rats_2} \subseteq G_{\rats_2}$ for the inertia group.  
Let $\ol{f}: \ol{Y} \to \ol{X}$ be the special fiber of $f^{st}$ (called the \emph{stable reduction} of $f$).  We focus on the monodromy action
of the inertia group $I_{\rats_2}$ on $\ol{f}$.  Throughout this section we write $v$ for the valuation on $K$ satisfying $v(2) = 1$.
We will allow finite extensions of $K$ as needed.

The following proposition is the result that underlies our entire computation.
\begin{proposition}[\cite{Ob:fm}, Lemma 7.8]\label{P2fermatreduction}
\begin{enumerate}[(i)]
\item There is exactly one irreducible component $\ol{X}_b$ of $\ol{X}$ above which $\ol{f}$ is generically \'{e}tale.  
\item Furthermore, $\ol{f}$ is \'{e}tale above $\ol{X}_b^{sm}$, i.e., the smooth points of $\ol{X}$ that lie on $\ol{X}_b$.  
\item Let $$d = \frac{a}{a+b} + \frac{\sqrt{2^nbi}}{(a+b)^2}.$$ and extend $K$ finitely (if necessary) so that $K$ contains $d$, as well as
an element $e$ such that 
$v(e) = n - \frac{s}{2} + \frac{1}{2}$.
Here, $i$ can be either square root of $-1$ and $\sqrt{2^nbi}$ can be either square root of $2^nbi$. 
Then, in terms of the coordinate $x$, the $K$-points of $X$ that specialize to $\ol{X}_b^{sm}$ form a closed disc of radius $|e|$ centered at $d$.
\item For each $k$-point $\ol{u}$ of $\ol{X}_b^{sm}$, 
the $K$-points of $X$ that specialize to $\ol{u}$ form an open disc of radius $|e|$.
\end{enumerate}
\end{proposition}

\begin{remark}\label{Rcoleman}
The result \cite[Lemma 7.8]{Ob:fm} is more general, in that it proves an analogous statement when $2$ is replaced by any prime $p$. Such a 
result was already shown in \cite{Co:fm} when $p$ is an odd prime (with some restrictions in the case $p=3$).  In \cite[Lemma 7.8]{Ob:fm}, $k$ is 
assumed to be algebraically closed, but as long as we restrict to $k$-points in Proposition \ref{P2fermatreduction}(iv), everything works.
\end{remark}

For any $K$-point $w$ in the closed disc from Proposition \ref{P2fermatreduction}(iii), write $\ol{w}$
for its specialization to $\ol{X}_b$, which is a $k$-point.  For such a $w$, if $t$ is defined by $x = w + et$, then 
$\hat{\mc{O}}_{X^{st}, \ol{w}} = R[[t]],$ where $\hat{\mc{O}}_{X^{st}, \ol{w}}$ is the completion of the local ring $\mc{O}_{X^{st}, \ol{w}}$ at its
maximal ideal.  The variable $t$ is called a \emph{parameter} for $\hat{\mc{O}}_{X^{st}, \ol{w}}$.  One thinks of 
$R[[t]]$ as the ring of functions on the open unit disc $|t| < 1$, which corresponds to the open disc $|x - w| < |e|$.

For $\gamma \in I_{\rats_2}$, let $\chi(\gamma) \in \ints_2^{\times}$ be the cyclotomic character applied to $\gamma$.
Maintain the notation $d$ from Proposition \ref{P2fermatreduction}. 

\begin{lemma}\label{Lgaloisaction}
Fix $\gamma \in I_{\rats_2}$.
Let $a'$ (resp.\ $b'$) be the integer between $0$ and $2^n -1$ congruent to $\chi(\gamma)a$ (resp.\ $\chi(\gamma)b$) modulo $2^n$.
Let $$d' = \frac{a'}{a'+b'} + \frac{\sqrt{2^nb'i}}{(a'+b')^2}$$
(here $i$ is the same square root of $-1$ chosen in the definition of $d$, but $\sqrt{2^nb'i}$ can be either choice of square root).
Then we have $\ol{\gamma(d)} = \ol{d'}$.
\end{lemma}

\begin{proof}
We first claim that 
\begin{equation}\label{E0}
d' \equiv \frac{a}{a+b} + \chi(\gamma)^{-3/2}\frac{\sqrt{2^nbi}}{(a+b)^2} \pmod{2^n},
\end{equation}
where $\sqrt{2^nbi}$ is chosen as in the definition of $d$, as long as the square root of $\chi(\gamma)$ is chosen correctly.
One verifies easily that 
\begin{equation}\label{E1}
\frac{a'}{a'+b'} \equiv \frac{a}{a+b} \pmod{2^n}.
\end{equation} 
One also sees easily that 
\begin{equation}\label{E2}
\frac{\sqrt{2^nb'i}}{(a'+b')^2} \equiv \frac{1}{\chi(\gamma)^2}\frac{\sqrt{2^nb'i}}{(a+b)^2} \pmod{2^n}.
\end{equation}
Now, 
\begin{equation}\label{E3}
\sqrt{2^nb'i} = \sqrt{2^n(\chi(\gamma)b + r)i} = \chi(\gamma)^{1/2}\sqrt{2^nbi}\sqrt{1 + \frac{r}{\chi(\gamma)b}},
\end{equation} where
$r$ is some integer divisible by $2^n$, where $\sqrt{1 + \frac{r}{\chi(\gamma)b}}$ is chosen to be no further from $1$ than from $-1$, where $\sqrt{2^nbi}$ is
chosen as in the definition of $d$, and where $\chi(\gamma)^{1/2}$ is 
chosen to make the equality work.  But $v\left(\frac{r}{\chi(\gamma)b}\right) \geq s$, and thus
$$v\left(\sqrt{1 + \frac{r}{\chi(\gamma)b}} - 1 \right) \geq s-1 \geq \frac{s}{2}$$ (recall that we assume $s \geq 2$).
Since $v(\sqrt{2^nbi}) = n - \frac{s}{2}$, it follows from (\ref{E3}) that $$\sqrt{2^nb'i} \equiv \chi(\gamma)^{1/2}\sqrt{2^nbi} \pmod{2^n}.$$
Combining this with (\ref{E1}) and (\ref{E2}) proves the claim.

Now, $$\gamma(d) = \frac{a}{a+b} + \zeta_{\gamma} \frac{\sqrt{2^nbi}}{(a+b)^2},$$ where $\zeta_{\gamma}$ is a fourth root 
of unity that depends on $\gamma$.  
In particular, 
$$\zeta_{\gamma} = \begin{cases} \pm i & \chi(\gamma) \equiv 3 \pmod{4} \\ \pm 1  & \chi(\gamma) \equiv 1 \pmod{4}.\end{cases}$$
In both cases, one computes that $\zeta_{\gamma} \equiv \chi(\gamma)^{-3/2} \pmod{2}$. 
So 
\begin{equation}\label{E4}
\gamma(d) \equiv \frac{a}{a+b} + \chi(\gamma)^{-3/2}\frac{\sqrt{2^nbi}}{(a+b)^2} \pmod{2^{n-\frac{s}{2} + 1}}.
\end{equation}
Combining this with (\ref{E0}), we obtain that $$\gamma(d) \equiv d' \pmod{2^{\min(n, n - \frac{s}{2} +1)}}.$$
Since $s \geq 2$, this implies $$\gamma(d) \equiv d' \pmod{2^{n - \frac{s}{2} + 1}}.$$
By Proposition \ref{P2fermatreduction}(iv), $\gamma(d)$ and $d'$ specialize to the same point, and we are done.
\end{proof}

\begin{remark}\label{Rdval}
Note that $v(d) = v(d') = 0$ and $v(d-1) = v(d'-1) = n-s.$
\end{remark}

Combining Proposition \ref{P2fermatreduction} and Lemma \ref{Lgaloisaction}, and using the definitions of $d$, $d'$, $e$, and $\gamma$ therein, we obtain:
\begin{corollary}\label{Cparameters}
If $x = d+et$, then $t$ is a parameter for $\Spec \hat{\mc{O}}_{X^{st}, \ol{d}}$.  Likewise, if
$x = d' + et'$, then $t'$ is a parameter for $\Spec \hat{\mc{O}}_{X^{st}, \ol{\gamma(d)}}$.
\end{corollary}

\subsection{Differential forms}\label{Sdiffforms}
Maintain the notation of \S\ref{Smonodromy}, including $d$, $d'$, $e$, and $\gamma$.  
All De Rham cohomology groups will be assumed to have coefficients in $K$.

As in \S\ref{Sfrob}, let $q = (\rho, \sigma, \tau) \in (\rats/\ints)^3,$ such that $\rho+\sigma+\tau = 0$.  Furthermore, suppose 
$\fracpart{\rho} = \frac{a}{2^n}$ with $a$ odd and $\fracpart{\sigma} = \frac{b}{2^n}$ with $1 \leq v(b) \leq n-2$.  
Set $\epsilon_q  = \fracpart{\rho} + \fracpart{\sigma} + \fracpart{\tau} - 1$.
Let $F_{2^n}$ be the Fermat curve given by $u^{2^n} + v^{2^n} = 1$, defined over $\rats_2$, and let $J_{2^n}$ be its Jacobian. 
Let $\omega_{2^n, q}$ be the element of $H^1_{DR}(F_{2^n}) \cong H^1_{DR}(J_{2^n})$ given by the differential form 
$$\eta_{2^n, q} = 2^n\fracpart{\rho+\sigma}^{\epsilon_q}u^av^b\frac{v}{u}d\left(\frac{u}{v}\right).$$
Recall that this is the pullback of a cohomology class $\omega_q$ on a
rational factor $J_q$ of $J_{2^n}$.
One can rewrite $\eta_{2^n, q}$ as $$\fracpart{\rho+\sigma}^{\epsilon_q}u^{a-2^n}v^{b-2^n}d(u^{2^n})$$ (cf.\ \cite[(1.2)]{Co:fm}).
Making the substitution $y = u^av^b$ and $x = u^{2^n}$ shows that $\eta_{2^n, q}$ (and thus
$\omega_{2^n, q}$) descends to the curve $Y$ given by the equation $y^{2^n} = x^a(x-1)^b$ (which we will also call $F_{2^n, a, b}$), 
and is given in $(x,y)$-coordinates by

$$\eta_{2^n, q} = \frac{\fracpart{\rho+\sigma}^{\epsilon_q}}{x(1-x)}ydx.$$

If $\gamma \in I_{\rats_2}$, then $\fracpart{\gamma \rho} = \frac{a'}{2^n}$ and $\fracpart{\gamma \sigma} = \frac{b'}{2^n}$, where $a'$ and $b'$ are as in Lemma \ref{Lgaloisaction}.  Letting $\gamma \in I_{\rats_2}$ act on $(\rats/\ints)^3$ diagonally via the cyclotomic character, 
we define $\eta_{2^n, \gamma q}$, $\omega_{2^n, \gamma q}$, and $\omega_{\gamma q}$ as above.  Now,
$\eta_{2^n, \gamma q}$ (and thus $\omega_{2^n, \gamma q}$) descends to the curve $F_{2^n, a', b'}$ given by the equation $(y')^{2^n} = x^{a'}(x-1)^{b'}$, 
where $y' = u^{a'}v^{b'}$. 
Then $\eta_{2^n, \gamma q}$ is given in $(x, y')$-coordinates by

$$\eta_{2^n, \gamma q} = \frac{\fracpart{\gamma \rho + \gamma \sigma}^{\epsilon_{\gamma q}}}{x(1-x)}y'dx.$$

Note that we can identify $F_{2^n, a', b'}$ with $F_{2^n, a, b}$ via $y' = y^hx^j(1-x)^k$, where $h$, $j$, and $k$ are such that
$a' = ha + 2^nj$ and $b' = hb + 2^nk$.

Recall from (\ref{Epullback}) that, for each $\gamma \in I_{\rats_2}$, there exists $\beta_{\gamma}(q) \in K$ (after a possible finite extension of $K$) 
such that $\gamma^*\omega_{2^n, q} = \beta_{\gamma}(q)\omega_{2^n, \gamma q}$ in $H^1_{DR}(J_{2^n})$.  We will compute $\beta_{\gamma}(q)$
by viewing $\omega_{2^n, q}$ and $\omega_{2^n, \gamma q}$ as cohomology classes on $F_{2^n, a, b} = F_{2^n, a', b'}$.

The following proposition relies on calculations from \S\ref{Scomputations}.
\begin{proposition}[cf.\ \cite{Co:fm}, Corollary 7.6]\label{Pdiffformratio}
We have $$v(\beta_{\gamma}(q)) = v(\fracpart{\rho})(\fracpart{\rho} - \fracpart{\gamma \rho}) + v(\fracpart{\sigma})(\fracpart{\sigma} - \fracpart{\gamma \sigma})
+ v(\fracpart{\tau})(\fracpart{\tau} - \fracpart{\gamma \tau}).$$
\end{proposition}

\begin{proof}
We work with the representatives $\eta_{2^n, q}$ and $\eta_{2^n, \gamma q}$ of $\omega_{2^n, q}$ and $\omega_{2^n, \gamma q}$ on the curve
$F_{2^n, a, b} = F_{2^n, a', b'}$.

If $x = d+et$, then Proposition \ref{Pasymptoticvaluation} defines (after a possible finite extension of $R$) a power series
$\alpha(t) \in R[[t]]$ such that $$\alpha(t)^{2^n} = x^a(x-1)^bd^{-a}(d-1)^{-b}$$ (Remark \ref{Rreminder}).  Corollary \ref{Cdiffform} defines
$\tilde{\alpha}(t) = \frac{d(1-d)\alpha(t)}{x(1-x)} \in R[[t]]$ (after substituting $x = d+et$), and shows that the valuation of the coefficient of 
$t^{\ell}$ in $\tilde{\alpha}(t)$ is 
$\frac{1}{2}S(\ell)$, the number of ones in the base $2$ expansion of $\ell$.  Since $y^{2^n} = x^a(x-1)^b$, we have
\begin{equation}\label{Efirstomega}
\eta_{2^n, q} = \frac{\sqrt[2^n]{d^a(d-1)^b}\fracpart{\rho + \sigma}^{\epsilon_q} \tilde{\alpha}(t)}{d(1-d)} \, e\, dt 
= \mu d^{\fracpart{\rho} - 1} (d-1)^{\fracpart{\sigma} - 1} \fracpart{\rho + \sigma}^{\epsilon_q} \tilde{\alpha}(t) e\, dt,
\end{equation}
where $\mu$ is some root of unity and $d^{\fracpart{\rho}}$, $(d-1)^{\fracpart{\sigma}}$. are calculated using some choices of $2^n$th roots.

Likewise, letting $d'$ be as in Lemma \ref{Lgaloisaction} and setting $x' = d' + et'$, we have
\begin{equation}\label{Esecondomega}
\eta_{2^n, \gamma q} 
= \mu' (d')^{\fracpart{\gamma \rho}-1} (d'-1)^{\fracpart{\gamma \sigma}-1} \fracpart{\gamma \rho + \gamma \sigma}^{\epsilon_{\gamma q}} 
\tilde{\alpha}'(t') e\, dt',
\end{equation} where $\mu'$ is some root of unity, and $\tilde{\alpha}'(t')$ is some power series in $t'$ whose coefficients have the \emph{same} valuations
as the coefficients of $\tilde{\alpha}(t)$ (Remark \ref{Rindependence}).

By Corollary \ref{Cparameters}, $t$ (resp.\ $t'$) is a parameter for $\Spec \hat{\mc{O}}_{X^{st}, \ol{d}}$ 
(resp.\ $\Spec \hat{\mc{O}}_{X^{st}, \ol{\gamma(d)}}).$  Since the map $Y^{st} \to X^{st}$ is completely split above $\ol{d}$ (Proposition \ref{P2fermatreduction}(ii)), 
we can also view
$t$ as a parameter for $\Spec \hat{\mc{O}}_{Y^{st}, \ol{u}}$ for any point $\ol{u} \in \ol{Y}$ above $\ol{d}$.  Then $t'$ can be viewed as
a parameter for $\Spec \hat{\mc{O}}_{Y^{st}, \gamma(\ol{u})}$.
Write $\eta_{2^n, q} = \sum_{\ell=0}^{\infty} z_\ell t^\ell dt$ and  
$\eta_{2^n, \gamma q} = \sum_{\ell=0}^{\infty} z'_\ell (t')^\ell dt'$.
By \cite[Theorem 4.1]{Co:fm} (setting $q=1$ in that theorem), 
$$v(\beta_{\gamma}(q)) = \lim_{i \to \infty} v\left(\frac{z_{\ell_i}}{z'_{\ell_i}}\right),$$ where $\ell_i$ is any sequence such that 
$\lim_{i \to \infty} v(z_{\ell_i}) - v(\ell_i + 1) = -\infty$. Take $\ell_i = 2^i - 1$.  Then, by  Remark \ref{Rdval}, Corollary \ref{Cdiffform}, 
(\ref{Efirstomega}), and (\ref{Esecondomega}), we have 
$$v(z_{\ell_i}) =  (n-s) (\fracpart{\sigma} - 1) - n \epsilon_q + \frac{i}{2} + (n - \frac{s}{2} + \frac{1}{2})$$ and
$$v(z'_{\ell_i}) =  (n-s) (\fracpart{\gamma \sigma} - 1) - n \epsilon_{\gamma q} + \frac{i}{2} + (n - \frac{s}{2} + \frac{1}{2}).$$
So $$v(\beta_{\gamma}(q)) = (n-s)(\fracpart{\sigma} - \fracpart{\gamma \sigma}) + n (\epsilon_{\gamma q} - \epsilon_q).$$ 
Some rearranging shows that this is equal to
$$n(\fracpart{\gamma \rho} - \fracpart{\rho}) + s(\fracpart{\gamma \sigma} - \fracpart{\sigma}) + n(\fracpart{\gamma \tau} - \fracpart{\tau}),$$
which is equal to the expression in the proposition.
\end{proof}

\subsection{Finishing the product formula}\label{Sformula}

If $\gamma \in I_{\rats_2}$ and $r \in \rats/\ints$, then let
$w_{2, \gamma}(r) = w_2(a_r) - w_2(a_{\gamma r})$, where the terms on the right hand side are defined in \S\ref{Sfrob}.  The following result is an important 
consequence of Proposition \ref{Pdiffformratio}.
\begin{corollary}\label{Ctransition}
Let $\gamma \in I_{\rats_2}$.
If $q = (\rho, \sigma, \tau) \in (\rats/\ints)^3$ with $\rho + \sigma + \tau = 0$, 
and none of $\fracpart{\rho}$, $\fracpart{\sigma}$, or $\fracpart{\tau}$ is $\frac{1}{2}$, then 
$w_{2, \gamma}(\rho) + w_{2, \gamma}(\sigma) + w_{2, \gamma}(\tau) = 0$.
\end{corollary}

\begin{proof} 
This has already been proven in \cite[Lemme III.2.5]{Co:pv} when any of $\rho$, $\sigma$, or $\tau$ is in $\ints_{(2)}/\ints$, so we assume 
otherwise.  Furthermore, \cite[Lemme III.2.6]{Co:pv} states that $w_{2, \gamma}(\alpha) = w_{2, \gamma}(\alpha')$ whenever 
$\alpha - \alpha' \in \ints_{(2)}/\ints$.  For each $\alpha \in (\rats/\ints)$, there is a unique $\alpha' \in \rats/\ints$ such that 
$\alpha - \alpha' \in \ints_{(2)}/\ints$ and $\fracpart{\alpha'} = \frac{j}{k}$, where $k$ is a power of $2$.  
Furthermore, if $\rho + \sigma + \tau = 0$, then
$\rho' + \sigma' + \tau' = 0$.  So we may assume that the denominators of $\rho$, $\sigma$, and $\tau$ are powers of $2$.

Let $n$ be minimal such that $\fracpart{\rho} = \frac{a}{2^n}$, $\fracpart{\sigma} = \frac{b}{2^n}$, and $\fracpart{\tau} = \frac{c}{2^n}$, 
with $a$, $b$, $c \in \ints$.  Then $n \geq 3$.
Assume without loss of generality that $v_2(b) \geq \max(v_2(a), v_2(c))$.  Then $a$ and $c$ must be odd, and $1 \leq v(b) \leq n-2$ 
(cf.\ \S\ref{Sdiffforms}---recall that 
we assume that $\fracpart{\sigma} \notin \{0, \frac{1}{2}\}$).

One can then copy the proof of \cite[Lemme III.2.5]{Co:pv}, with our Proposition \ref{Pdiffformratio} substituting for \cite[Corollary 7.6]{Co:fm}.
In more detail, $$w_{2, \gamma}(\rho) + w_{2, \gamma}(\sigma) + w_{2, \gamma}(\tau) = w_2(b_q) - w_2(b_{\gamma q}).$$  
Using the definitions from \S\ref{Sfrob},  
this is equal to $$V_2(b_{\gamma q}) - V_2(b_q) + v_2(\omega_q) - v_2(\omega_{\gamma q}),$$ which is equal to
$$V_2(b_{\gamma q}) - V_2(b_q) - v(\beta_{\gamma}(q)),$$ by Lemma \ref{Laux}.  By Proposition \ref{Pdiffformratio} and the fact that $\rho_{(2)}$, 
$(\gamma \rho)_{(2)}$, $\sigma_{(2)}$, $(\gamma \sigma)_{(2)}$, $\tau_{(2)}$, and $(\gamma \tau)_{(2)}$ are all zero (Equation (\ref{Epdef})), this is equal to zero.
\end{proof}

\begin{corollary}\label{Czero}
For all $\gamma \in I_{\rats_2}$ and $r$ in $\rats/\ints$, we have $w_{2, \gamma}(r) = 0$.  
\end{corollary}

\begin{proof}
If $\fracpart{r} \in \{0, \frac{1}{2}\}$, then $r = \gamma r$, thus $w_{2, \gamma}(r) = 0$ by definition. 
We also have $w_{2, \gamma}(-r) = -w_{2, \gamma}(r)$ for all $r \in \rats/\ints$ (this follows from plugging 
$(\rho, \sigma, \tau) = (r, -r, 0)$ into Corollary \ref{Ctransition}, unless $\fracpart{r} = \frac{1}{2}$, in which case it is obvious).
Plugging any triple $(a, b, -(a+b))$ 
into Corollary \ref{Ctransition} then shows that $$w_{2, \gamma}(a) + w_{2, \gamma}(b) = w_{2, \gamma}(a+b),$$ as long as
none of $\fracpart{a}$, $\fracpart{b}$, or $\fracpart{a+b}$ is $\frac{1}{2}$.

We now claim that, if $k > 4$ is even, and if $a \in \rats/\ints$ satisfies $\fracpart{a} = \frac{1}{k}$, 
then $w_{2, \gamma}(ja) = jw_{2, \gamma}(a)$ for
$1 \leq j \leq \frac{k}{2} -1$ and for $\frac{k}{2} + 1 \leq j \leq k$.
Admitting the claim, we set $j = k$ to show that $w_{2, \gamma}(a) = 0$, which in turn shows that
$w_{2, \gamma}(ja) = 0$ for all $j$ above.  Since any $r \in ([0, 1) \cap \rats)\backslash \{\frac{1}{2}\}$ is the fractional
part of some such $ja$, the claim implies the corollary.

To prove the claim, we note by additivity of $w_{2, \gamma}$ that $w_{2, \gamma}(ja) = jw_{2, \gamma}(a)$ for $1 \leq j \leq \frac{k}{2} - 1$.  
By additivity again (using $(\frac{k}{2} - 1)a$ and $2a$, neither of which has fractional part $\frac{1}{2}$), 
we have $w_{2, \gamma}((\frac{k}{2} + 1)a) = (\frac{k}{2} + 1)w_{2, \gamma}(a)$.   
Then, additivity shows that $w_{2, \gamma}(ja) = jw_{2, \gamma}(a)$ for $\frac{k}{2} + 1 \leq j \leq k$.
\end{proof}

\begin{theorem}\label{Tmain}
We have $w_2(a) = 0$ for all $a \in \mc{CM}^{ab}$.
\end{theorem}

\begin{proof}
This follows from Corollary \ref{Czero} exactly as \cite[Corollaire III.2.7]{Co:pv} follows from \cite[Lemme III.2.6(i)]{Co:pv}.
\end{proof}

Theorem \ref{Tmain} completes the proof of Colmez's product formula when the field of complex multiplication is an abelian extension of $\rats$.

\begin{remark}\label{Ralready}
Colmez already proved Corollary \ref{Czero} when $r \in \frac{1}{8}\ints_{(2)}/\ints$ (\cite[Lemma III.2.8]{Co:pv}).  This was used to give a geometric proof of the 
Chowla-Selberg formula (\cite[III.3]{Co:pv}).
\end{remark}

\section{Computations}\label{Scomputations}
The results of this section are used only in the proof of Proposition \ref{Pdiffformratio}.

\subsection{Base $2$ expansions}\label{Sbase2}
Let $S(\ell)$ be the sum of the digits in the base $2$ expansion of $\ell$, or $\infty$ if $\ell \in \rats \backslash \{0, 1, 2, \ldots\}$.
It is clear that $S(\ell) = 1$ iff $\ell$ is an integer and a power of $2$.  Note also that if $\ell_1$ and $\ell_2$ are positive integers 
whose ratio is a power of $2$, then $S(\ell_1) = S(\ell_2)$.

\begin{lemma}\label{Lbase2}
If $\ell_1$ and $\ell_2$ are nonnegative integers, then $S(\ell_1 + \ell_2) \leq S(\ell_1) + S(\ell_2)$.  Equality never holds if $\ell_1 = \ell_2$.  Furthermore, if $\ell$ is a 
positive integer, there are
exactly $2^{S(\ell)} - 2$ ordered pairs of positive integers $(\ell_1, \ell_2)$ such that $\ell_1 + \ell_2 = \ell$ and $S(\ell_1) + S(\ell_2) = S(\ell)$.
\end{lemma}

\begin{proof}
The first two assertions are clear from the standard addition algorithm.  Now, for positive integers $\ell_1$ and $\ell_2$, we have $S(\ell_1 + \ell_2) = S(\ell_1) + S(\ell_2)$
exactly when no carrying takes place in the addition of $\ell_1$ and $\ell_2$ in base $2$.  This happens when $\ell_1$ is formed by taking a nonempty, proper subset
of the $1$'s in the base $2$ expansion of $\ell$, and converting them to zeros.  There are $2^{S(\ell)-2}$ such subsets, proving the lemma. 
\end{proof}

The following lemma gathers several elementary facts.  The somewhat strange phrasings will pay off in \S\ref{Sfermat}.  Notice that all inequalities are phrased in terms
of something being less than or equal to $\frac{1}{2}S(\ell)$.
\begin{lemma}\label{Lpossibilities}
Let $\ell$ be a positive integer.  
\begin{enumerate}[(i)]
\item $2S(\frac{\ell}{4}) - 2 \leq \frac{1}{2}S(\ell)$ iff $\ell \geq 4$ is a power of $2$.
\item $S(\frac{\ell}{2}) - 1 \leq \frac{1}{2}S(\ell)$ iff $S(\ell) \leq 2$ and $\ell$ is even.
\item There are exactly $2^{S(\ell)} - 2$ ordered pairs of positive integers $(\ell_1, \ell_2)$ such that $\ell_1 + \ell_2 = \ell$ and
$\frac{1}{2}S(\ell_1) + \frac{1}{2}S(\ell_2) \leq \frac{1}{2}S(\ell)$.
\item If $\ell_1$ and $\ell_2$ are distinct positive integers such that $2(\ell_1 + \ell_2) = \ell$, then
$S(\ell_1) + S(\ell_2) - 1 \leq \frac{1}{2}S(\ell)$ iff $S(\ell_1) = S(\ell_2) = 1$ and $S(\ell) = 2$.
\item If $\ell_1$, $\ell_2$, and $\ell_3$ are positive integers, not all distinct, such that $\ell_1 + \ell_2 + \ell_3 = \ell$, then it is never the case that
$\frac{1}{2}S(\ell_1) + \frac{1}{2}S(\ell_2) + \frac{1}{2}S(\ell_3) \leq \frac{1}{2}S(\ell)$.
\item If $\ell_1$ and $\ell_2$ are distinct positive integers such that $\ell_1 + 3\ell_2 = \ell$, then it is never the case that
$\frac{1}{2}S(\ell_1) + \frac{3}{2}S(\ell_2) \leq \frac{1}{2}S(\ell)$.
\item If $\ell_1$, $\ell_2$, and $\ell_3$ are distinct positive integers such that $\ell_1 + \ell_2 + 2\ell_3 = \ell$, then it is never the case that
$\frac{1}{2}S(\ell_1) + \frac{1}{2}S(\ell_2) + S(\ell_3) \leq \frac{1}{2}S(\ell)$.
\item If $\ell_1$, $\ell_2$, $\ell_3$, and $\ell_4$ are distinct nonnegative integers such that $\ell_1 + \ell_2 + \ell_3 + \ell_4 = \ell$, then it is never the case that
$\frac{1}{2}S(\ell_1) + \frac{1}{2}S(\ell_2) + \frac{1}{2}S(\ell_3) + \frac{1}{2}S(\ell_4) + 1 \leq \frac{1}{2}S(\ell)$.
\end{enumerate}
\end{lemma}

\begin{proof}
Parts (i) and (ii) are easy, using that $S(\ell/4)$ and $S(\ell/2)$ are either equal to $S(\ell)$ or $\infty$.  Part (iii) follows from Lemma \ref{Lbase2}.
Part (iv) follows from that fact that $S(\ell) = S(\ell_1 + \ell_2) \leq S(\ell_1) + S(\ell_2)$.  Parts (v), (vi), (vii), and (viii) also follow from 
Lemma \ref{Lbase2}.
\end{proof}

\subsection{Power series}\label{Spower}
As in \S\ref{Sfermat}, 
let $f: Y \to X = \proj^1$ be the branched cover of smooth curves given birationally by $y^{2^n}=x^a(x-1)^b$ where $a$ is odd, $1 \leq v_2(b) \leq n-2$, and 
$0 < a, b < 2^n$. 
Throughout this section, we take $K/\rats_2$ to be a finite extension over which $f$ admits a stable model, and $R$ to be the
ring of integers of $K$.  We will take further finite extensions of $K$ and $R$ as necessary.  The valuation $v$ on $K$ (and any finite extension) is always normalized 
so that $v(2) = 1$.  Throughout this section, we fix a square root $i$ of $-1$ in $K$.
We let $f^{st}: Y^{st} \to X^{st}$ be the stable model of $f$, and $\ol{f}: \ol{Y} \to \ol{X}$ its stable reduction (\S\ref{Smonodromy}).  

Set $d = \frac{a}{a+b} + \frac{\sqrt{2^nbi}}{(a+b)^2}$, and $s:= n - v_2(b) \geq 2$.  
Let $\ol{d}$ be the specialization of $d$ in $\ol{X}$.  
Let $e$ be any element of $R$ with valuation $n- \frac{s}{2} + \frac{1}{2}$.  If $x = d+et$, then $t$ is a parameter
of $\hat{\mc{O}}_{X^{st}, \ol{d}}$ (Corollary \ref{Cparameters}).  
We set $$g(x) = x^a(x-1)^bd^{-a}(d-1)^{-b}.$$  Note that $g(d) = 1$.  
\begin{lemma}\label{Lnormalize}
Expanding out $g(x)$ in terms of $t$ yields an expression of the form
$$\gamma(t) := g(d+et) = \sum_{\ell=0}^{\infty} c_{\ell} t^\ell$$ where $c_0 = 1$, $v(c_2)=n$, $\frac{c_1^2}{c_2} \equiv 2^{n+1}i\pmod{2^{n+2}}$, and $v(c_{\ell}) > 
n + \frac{1}{2}S(\ell)$ for all $\ell \geq 3$.  In particular, $v(c_1) = n + \frac{1}{2}$.
\end{lemma}

\begin{remark}\label{Rfinite}
Of course, the ``series" above is actually just a polynomial.
\end{remark}

\begin{proof}
The claim at the beginning of the proof of the $p=2$ part of \cite[Lemma C.2]{Ob:fm} proves everything except the statement for $\ell \geq 3$.
The continuation of the proof of \emph{loc. cit.} leads to $$v(c_\ell) = n + 1 + \frac{\ell-2}{2}(s+1) - v(\ell) > n + \ell - 1 - v(\ell)$$ (recall, $s \geq 2$).    
It is easy to see that $\ell > 1 + v(\ell) + \frac{1}{2}S(\ell)$ for $\ell \geq 3$, from which the lemma follows.
\end{proof}

We wish to understand the $2^n$th root of $\gamma(t)$.  It turns out that it is easier to do this by first taking a $2^{n-2}$th root, and then a $4$th root.

\begin{lemma}\label{L4thpower}
After possibly replacing $R$ by a finite extension, 
the power series $\gamma(t) = \sum_{\ell=0}^{\infty} c_\ell t^\ell$ from Lemma \ref{Lnormalize} has a $2^{n-2}$nd root in $R[[t]]$ of the form
$$\delta(t) = \sum_{\ell=0}^{\infty} d_\ell t^\ell,$$
where $d_0 = 1$, $v(d_2) = 2$, $\frac{d_1^2}{d_2} \equiv 8i \pmod{16}$, and $v(d_\ell) > 2 + \frac{1}{2}S(\ell)$ for $\ell \geq 3$.  In particular, $v(d_1) = \frac{5}{2}$.
\end{lemma}

\begin{proof}
Let $w = \gamma(t) - 1$.  Binomially expanding $(1 + w)^{1/2^{n-2}}$ gives 
$$\delta(t) = 1 + \frac{w}{2^{n-2}} + \sum_{j=2}^{\infty} \binom{1/2^{n-2}}{j} w^j.$$
The valuation of $\binom{1/2^{n-2}}{j}$ is $$S(j) - j - j(n-2) = S(j) + j - jn.$$  
On the other hand, the valuation of $c_{\ell}$ (the coefficient of $t^{\ell}$ in $w$)  
is at least $n + \frac{1}{2}S(\ell) - \frac{1}{2}$ (Lemma \ref{Lnormalize}, the equality only holds for $\ell = 2$). 
So, by Lemma \ref{Lbase2}, the coefficient of $t^\ell$ in $w^j$ for $j \geq 2$ has valuation greater than 
$jn + \frac{1}{2} S(\ell) - \frac{j}{2}$ (equality could only occur if $\ell = 2j$, but in fact, does not, because 
$j(n + \frac{1}{2}S(2) - \frac{1}{2}) > jn + \frac{1}{2}S(2j) - \frac{j}{2}).$  
Combining everything, the coefficient of $t^{\ell}$ in $\binom{1/2^{n-2}}{j} w^j$ (for $j \geq 2$) has valuation greater than
$S(j) + \frac{j}{2} + \frac{1}{2}S(\ell)$, which is at least $2 + \frac{1}{2}S(\ell)$.

Thus, for the purposes of the lemma, we may replace $\delta(t)$ by $1 + \frac{w}{2^{n-2}}$.  The lemma then follows easily from Lemma \ref{Lnormalize}.
\end{proof}

\begin{proposition}\label{Pasymptoticvaluation}
After possibly replacing $R$ by a finite extension, 
the power series $\delta(t) = \sum_{\ell=0}^{\infty} d_\ell t^{\ell}$ from Lemma \ref{L4thpower} has a $4$th root in $R[[t]]$ of the form
$$\alpha(t) = \sum_{j=0}^{\infty} a_{\ell} t^{\ell},$$ where $a_0 = 1$, and $$a_\ell \equiv d_1^{\ell}(1 + i)^{S(\ell) -5\ell} \pmod{(1+i)^{S(\ell) + 1}}.$$ 
In particular, $v(a_\ell) = \frac{1}{2}S(\ell)$.
\end{proposition}

\begin{remark}\label{Rreminder}
Note that $\alpha(t)$ is a $2^n$th root of $$g(d+et) = x^a(x-1)^bd^{-a}(d-1)^{-b},$$ where $x = d+et$.
\end{remark}

\begin{proof}[of Proposition \ref{Pasymptoticvaluation}]
By Proposition \ref{P2fermatreduction}(ii), the stable model $f^{st}$ of $f$ splits completely above $\hat{\mc{O}}_{X^{st}, \ol{d}} = R[[t]]$.  
Thus, by \cite[Proposition 3.2.3 (2)]{Ra:ab}, $x^a(x-1)^b$ (when written in terms of $t$) is a $2^n$th power in 
$R[[t]]$.  This does not change when it is multiplied by the constant $d^{-a}(d-1)^{-b}$ (as long as we extend $R$ appropriately), 
so we see that $\alpha(t)$ lives in $R[[t]]$ (this can also be shown using an explicit computation with the binomial theorem).  

We have the equation
\begin{equation}\label{Etwoexpansions}
\left(1 + \sum_{\ell=1}^{\infty}a_\ell t^\ell\right)^4 \equiv 1 + \sum_{\ell=1}^{\infty}d_\ell t^\ell.
\end{equation}

We prove the proposition by strong induction, treating the base cases $\ell = 1$, $2$ separately.
Recall that $v(d_1) = \frac{5}{2}$ and $v(d_2) = 2$.
For $\ell=1$, we obtain from (\ref{Etwoexpansions}) that $d_1 = 4a_1$, so $$a_1 = \frac{d_1}{4} \equiv d_1(1+i)^{-4} \pmod{(1+i)^{2}}.$$
For $\ell=2$, we obtain $$d_2 = 4a_2 + 6a_1^2 = 4a_2 + \frac{3}{8}d_1^2,$$ so $a_2 = \frac{d_2}{4} - \frac{3}{32}d_1^2$.  
Using that $\frac{d_1^2}{d_2} \equiv 8i \pmod{16}$ (Lemma \ref{L4thpower}), 
one derives that $\frac{d_2}{4} \equiv \frac{d_1^2}{32i} \pmod{2}$.  Thus, $$a_2 \equiv (-i-3) \frac{d_1^2}{32} \equiv d_1^2(1+i)^{-9} \pmod{(1+i)^2},$$ proving the
proposition for $\ell = 2$.

Now, suppose $\ell > 2$.  Then (\ref{Etwoexpansions}) yields (setting $a_j = 0$ for any $j \notin \ints$, and with all $\ell_i$ assumed to be positive integers):
\begin{align*}
d_{\ell} &= 4a_{\ell} + 6a_{\ell/2}^2 + 4a_{\ell/3}^3 + a_{\ell/4}^4  
+ \sum_{\substack{\ell_1 +\ell_2 = \ell \\ \ell_1 < \ell_2}} 12 a_{\ell_1}a_{\ell_2} 
+ \sum_{\substack{\ell_1 +\ell_2 + \ell_3 = \ell \\ \ell_1 < \ell_2 < \ell_3}} 24 a_{\ell_1}a_{\ell_2}a_{\ell_3}\\ 
&+ \sum_{\substack{\ell_1 + 2\ell_2 = \ell \\ \ell_1 \ne \ell_2}} 12 a_{\ell_1}a_{\ell_2}^2   
+ \sum_{\substack{\ell_1 + 3\ell_2 = \ell \\ \ell_1 \ne \ell_2}} 4 a_{\ell_1}a_{\ell_2}^3   
+ \sum_{\substack{\ell_1 + \ell_2 + \ell_3 + \ell_4 = \ell \\ \ell_1 < \ell_2 < \ell_3 < \ell_4}} 24 a_{\ell_1}a_{\ell_2}a_{\ell_3}a_{\ell_4}\\
&+ \sum_{\substack{\ell_1 + \ell_2  + 2\ell_3 = \ell \\ \ell_1 < \ell_2, \ell_3 \ne \ell_1, \ell_3 \ne \ell_2}} 12 a_{\ell_1}a_{\ell_2}a_{\ell_3}^2 
+ \sum_{\substack{2\ell_1 + 2\ell_2 = \ell \\ \ell_1 < \ell_2}} 6 a_{\ell_1}^2a_{\ell_2}^2.
\end{align*}  
or
\begin{align*}
a_{\ell} &= -\frac{1}{4}d_{\ell} + \frac{3}{2}a_{\ell/2}^2 + a_{\ell/3}^3 + \frac{1}{4}a_{\ell/4}^4  
+ \sum_{\substack{\ell_1 +\ell_2 = \ell \\ \ell_1 < \ell_2}} 3 a_{\ell_1}a_{\ell_2} 
+ \sum_{\substack{\ell_1 +\ell_2 + \ell_3 = \ell \\ \ell_1 < \ell_2 < \ell_3}} 6 a_{\ell_1}a_{\ell_2}a_{\ell_3} \\ 
&+ \sum_{\substack{\ell_1 + 2\ell_2 = \ell \\ \ell_1 \ne \ell_2}} 3 a_{\ell_1}a_{\ell_2}^2   
+ \sum_{\substack{\ell_1 + 3\ell_2 = \ell \\ \ell_1 \ne \ell_2}}  a_{\ell_1}a_{\ell_2}^3   
+ \sum_{\substack{\ell_1 + \ell_2 + \ell_3 + \ell_4 = \ell \\ \ell_1 < \ell_2 < \ell_3 < \ell_4}} 6a_{\ell_1}a_{\ell_2}a_{\ell_3}a_{\ell_4} \\
&+ \sum_{\substack{\ell_1 + \ell_2  + 2\ell_3 = \ell \\ \ell_1 < \ell_2, \ell_3 \ne \ell_1, \ell_3 \ne \ell_2}} 3 a_{\ell_1}a_{\ell_2}a_{\ell_3}^2 
+ \sum_{\substack{2\ell_1 + 2\ell_2 = \ell \\ \ell_1 < \ell_2}} \frac{3}{2} a_{\ell_1}^2a_{\ell_2}^2.
\end{align*}  

Since we need only determine $a_{\ell}$ modulo $(1+i)^{S(\ell)+1}$, and all terms on the right hand side have half-integer valuation, 
we can ignore all terms with valuation greater than $\frac{1}{2}S(\ell)$.  Using the inductive hypothesis, along with Lemmas \ref{L4thpower} and 
\ref{Lpossibilities} (v), (vi), (vii), and (viii), we obtain

\begin{equation}\label{Ereduced}
a_{\ell} \equiv \frac{3}{2}a_{\ell/2}^2 + \frac{1}{4}a_{\ell/4}^4  
+ \sum_{\substack{\ell_1 +\ell_2 = \ell \\ \ell_1 < \ell_2}} 3 a_{\ell_1}a_{\ell_2} 
+ \sum_{\substack{2\ell_1 + 2\ell_2 = \ell \\ \ell_1 < \ell_2}} \frac{3}{2} a_{\ell_1}^2a_{\ell_2}^2 \pmod{(1+i)^{S(\ell) + 1}}.
\end{equation}

If $\ell$ is a power of $2$ (i.e., $S(\ell) = 1$), then by Lemma \ref{Lpossibilities} (i)--(iv), the induction hypothesis, and (\ref{Ereduced}), we have 
\begin{align*}
a_{\ell} \equiv \frac{3}{2}a_{\ell/2}^2 + \frac{1}{4}a_{\ell/4}^4 &\equiv d_1^{\ell}\left(\frac{3}{2}(1+i)^{2 - 5\ell} + \frac{1}{4}(1+i)^{4 - 5\ell}\right)\\
&\equiv d_1^{\ell}(3i - 1)(1+i)^{-5\ell} \equiv d_1^{\ell} (1+i)^{1- 5\ell} \pmod{(1+i)^2},
\end{align*}
thus proving the proposition for such $\ell$.

For all other $\ell$, we have (by Lemma \ref{Lpossibilities} (i), the induction hypothesis, and (\ref{Ereduced})) that
\begin{equation}\label{Ereduced2}
a_{\ell} \equiv \frac{3}{2}a_{\ell/2}^2  
+ \sum_{\substack{\ell_1 +\ell_2 = \ell \\ \ell_1 < \ell_2}} 3 a_{\ell_1}a_{\ell_2} 
+ \sum_{\substack{2\ell_1 + 2\ell_2 = \ell \\ \ell_1 < \ell_2}} \frac{3}{2} a_{\ell_1}^2a_{\ell_2}^2 \pmod{(1+i)^{S(\ell) + 1}}.
\end{equation}
By Lemma \ref{Lpossibilities} (ii) and (iv), the first and last terms matter only when $S(\ell) = 2$ and $\ell$ is even, in which case their 
combined contribution is
$$3d_1^{\ell}(1+i)^{4-5\ell} \pmod{(1+i)^3},$$ which is trivial.  
So in any case, we need only worry about the middle term.  By Lemma \ref{Lpossibilities} (iii), the middle
term is the sum of $2^{S(\ell) - 1} - 1$ subterms, each congruent to $$3d_1^{\ell}(1+i)^{S(\ell) - 5\ell} \pmod{(1+i)^{S(\ell) + 1}}.$$  
Since $S(\ell) \geq 2$, this sum is in 
turn congruent to $$d_1^{\ell}(1+i)^{S(\ell) - 5\ell} \pmod{(1+i)^{S(\ell) + 1}},$$ proving the proposition.
\end{proof}

\begin{corollary}\label{Cdiffform}
In the notation of Proposition \ref{Pasymptoticvaluation}, the power series 
$$\tilde{\alpha}(t) := \frac{d(1-d)\alpha(t)}{x(1-x)} = \frac{d(1-d)\alpha(t)}{(d+et)(1-d-et)}$$
has the form $$\sum_{i=0}^{\infty} \tilde{a}_{\ell} t^{\ell},$$ where $\tilde{a}_0 = 1$ and $v(\tilde{a}_{\ell}) = v(a_{\ell}) = \frac{1}{2}S(\ell)$ for all $\ell$.
\end{corollary}

\begin{proof}
Recall that $v(1-d) = n-s$ (Remark \ref{Rdval}), 
that $v(e) = n - \frac{s-1}{2}$, and that we assume $2 \leq s \leq n-1$.  Set $\mu = -\frac{e}{d}$ and $\nu = \frac{e}{1-d}$.  Then $v(\mu) = 
n-\frac{s-1}{2} > 1$ and $v(\nu) = \frac{s+1}{2} > 1$.  Expanding $\tilde{\alpha}(t)$ out as a power series yields
$$\tilde{\alpha}(t) = \alpha(t)(1 + \mu t + \mu^2t^2 + \cdots)(1 + \nu t + \nu^2 t^2 + \cdots) = \alpha(t)(1 + \xi_1t + \xi_2 t^2 + \cdots),$$
where $v(\xi_{\ell}) > \ell$ for all $\ell$.  The constant term is $1$, so $\tilde{a}_0 = 1$.  The coefficient of $t^{\ell}$ is 
$$\tilde{a}_{\ell} = a_{\ell} + \xi_{\ell} + \sum_{j=1}^{\ell-1} a_{\ell-j}\xi_j.$$
We know $v(a_{\ell}) = \frac{1}{2}S(\ell)$.  We have seen that $v(\xi_{\ell}) > \ell > \frac{1}{2}S(\ell)$.  Also, for $1 \leq j \leq \ell-1$, we have
$$v(a_{\ell-j}\xi_j) > \frac{1}{2}S(\ell-j) + j > \frac{1}{2}S(\ell-j) + \frac{1}{2}S(j) \geq \frac{1}{2}S(\ell).$$  By the non-archimedean property, we conclude that
$v(\tilde{a}_{\ell}) = \frac{1}{2}S(\ell)$.
\end{proof}

\begin{remark}\label{Rindependence}
Note that $v(\tilde{a}_{\ell})$ does not depend on $a$ or $b$.
\end{remark}

\begin{acknowledgements}
I thank Pierre Colmez, Dick Gross, Johan de Jong, Melissa Liu, and Shouwu Zhang for useful conversations.  In particular, I thank Michel Matignon for pointing me in the direction of this question.  I also thank the referee for useful expository suggestions.
\end{acknowledgements}

\end{document}